\documentclass[12pt]{amsart}

\usepackage{amsfonts, amsmath, amscd}
\usepackage[psamsfonts]{amssymb}

\usepackage{amssymb}

\usepackage[usenames]{color}

\newtheorem{theorem}{Theorem}%[section]
\newtheorem{lemma}[theorem]{Lemma}
\newtheorem{corollary}[theorem]{Corollary}
\newtheorem{question}[theorem]{Question}
\newtheorem{remark}[theorem]{Remark}

\numberwithin{equation}{section}

\newcommand{\NN}{\mathbb{N}}
\newcommand{\ZZ}{\mathbb{Z}}
\newcommand{\TT}{\mathbb{T}}

\newcommand{\w}{\omega}

\newcommand{\RR}{\mathbb{R}}

\newcommand{\A}{\mathcal{A}}

\renewcommand{\phi}{\varphi}

\input xy
\xyoption{all}

\title[An Open Mapping Theorem]{An Open Mapping Theorem}

\author[S. S. Gabriyelyan]{Saak S. Gabriyelyan}
\address{Department of Mathematics, Ben-Gurion University of the Negev,
Beer-Sheva P.O. 653, Israel}
\email{saak@math.bgu.ac.il}

\author[S. A. Morris]{Sidney A. Morris}
\address{Faculty of Science and Technology, %\linebreak
Federation University Australia, PO Box 663, Ballarat, Victoria, 3353,  Australia \& %\linebreak
Department of Mathematics and Statistics,%\linebreak
 La~Trobe University, Melbourne, Victoria, 3086, Australia}
\email{morris.sidney@gmail.com}

\subjclass[2000]{Primary: 46A30; Secondary: 54H11, 54E52.}

\keywords{ pro-Lie group, locally compact abelian group, Open Mapping Theorem}

\begin{document}

\begin{abstract}
It is proved that any surjective morphism $f: \ZZ^\kappa \to K$ onto a locally compact group $K$ is open for every cardinal $\kappa$. This answers a question posed by Karl Heinrich Hofmann and the second author.
\end{abstract}

\maketitle

\section{Introduction}
In this paper we assume that all topological groups are  Hausdorff and abelian.
In the literature it is common to ask whether a surjective continuous homomorphism  
 $f: G\to K$ of a topological group $G$ onto a topological group $K$ is an open mapping. Positive results in this direction are  known as ``Open Mapping Theorems'' in the literature in Functional Analysis and Topological Algebra (see, for example, \cite[2.25]{fabian-10} for Banach spaces, \cite{KoTa} for Polish groups and \cite[9.60]{HM-proLie} and \cite{hm} for pro-Lie groups).  Most results of this type  impose a countability condition on $G$. Indeed, if $K$ is any countable non-discrete group or an infinite compact one and $G:=K_d$ is the group $K$ endowed with the discrete topology, then the identity map $i:G\to K$  is not open. Noting that for every uncountable cardinal $\kappa$ the totally disconneced abelian group $G=\ZZ^\kappa$ is neither a  Polish group nor a locally compact group, K.H.~Hofmann and the second author posed the following question, see Question 5 in \cite{HM}: {\em Is a surjective morphism $f: \ZZ^\kappa \to K$ onto a compact group open for every cardinal $\kappa$?}
We answer this question in the affirmative.

We will use the following notation and terminology.
For a topological group $K$, we denote by $K_0$  the connected component of the identity.
A topological group $K$ is called {\em almost connected} \cite{HM-proLie} if the quotient group $K/K_0$ is compact. A topological group $G$ is called a {\em pro-Lie group}  \cite{HM-proLie} if it is a closed subgroup of a product of finite-dimensional Lie groups. So the group $\ZZ^\kappa$ is a non-almost connected pro-Lie group. Every compact group is an almost connected pro-Lie group.

 We denote by $\mathcal{CDA}$ the class of all totally disconnected abelian groups $G$ such that for every open subgroup $H$ of $G$ the quotient group $G/H$ is countable. Note that the class $\mathcal{CDA}$  is closed  under taking subgroups, Hausdorff quotient groups and arbitrary  products (with the Tychonoff topology).

\section{Results}
The first lemma is an immediate consequence of Proposition 5.43 of \cite{HM-proLie}.

\begin{lemma} \label{l:H-M-connected}
Every non-totally disconnected abelian pro-Lie group $K$ has the circle group as a quotient group.
\end{lemma}

\begin{lemma} \label{l:H-M-disconnected}
Let $G\in \mathcal{CDA}$. If there is a surjective morphism $f: G \to K$ onto a  pro-Lie group $K$, then $K$ also belongs to the class $\mathcal{CDA}$.
\end{lemma}

\begin{proof}
Suppose  $K$ is not totally disconnected.   Then, by Lemma \ref{l:H-M-connected}, there is a continuous homomorphism ${\bar f}$ from $G$ onto $\TT$. Let $U$ be any small neighborhood of the identity in the circle group $\TT$. Note that $U$ contains no non-trivial subgroups of $\TT$. As ${\bar f}^{-1}(U)$ is an open neighborhood of zero of $G$, ${\bar f}^{-1}(U)$ contains an open subgroup $H$ of $G$ such that $G/H$ is countable. So $f(H)=\{ 0\}$ and hence $f(G)=\TT$ is also countable, a contradiction to our supposition. Hence $G$ is totally disconnected.

It remains to show that for every open subgroup $H$ of $K$ the quotient group $K/H$ is countable. This  follows from the fact that $G/f^{-1}(H)$ is countable and $f$ is surjective.
\end{proof}

\begin{lemma} \label{l:H-M-tot-dis}
Let $K$ be an almost connected abelian  pro-Lie  group which is either totally disconnected or a  torsion group. Then $K$ is compact.
\end{lemma}
\begin{proof}
If $K$ is totally disconnected, then $K=K/K_0$ is compact by definition.

Assume that $K$ is torsion. By Theorem 5.20 of \cite{HM-proLie}, a torsion abelian connected pro-Lie group is compact. So the connected component $K_0$ of $K$ is compact.  But the almost connected group $K$ by definition is then an extension of a  compact group by a compact group.  As compactness is a three space property for topological groups, $K$ is compact.
\end{proof}

For every $m>1$ and cardinal number $\kappa$, the group  $\ZZ^\kappa/m\ZZ^\kappa =\ZZ(m)^\kappa$ is compact. Being motivated by this fact, we denote by $\mathcal{CDA}_k$ the class of all groups $G\in\mathcal{CDA}$ for which $G/G_m$ is compact for every natural number $m>1$, where $G_m := \mathrm{cl}_G (mG)$, the closure in $G$ of $mG$. So $\ZZ^\kappa\in \mathcal{CDA}_k$. Note that the group $G/G_m$ has order $\leq m$, for every $m>1$.  We note also that the class $\mathcal{CDA}_k$ is closed  under taking  Hausdorff quotient groups and arbitrary products.

\begin{lemma} \label{l:H-M-bounded}
Let $G\in \mathcal{CDA}_k$. If $f: G \to K$  is a surjective continuous homomorphism onto an  almost connected torsion pro-Lie group $K$, then $f$ is an open mapping.
\end{lemma}

\begin{proof}
By Lemma \ref{l:H-M-tot-dis} we can suppose that $K$ is a compact abelian group.

Since $K$ is torsion, there is an $m\in\NN$ such that $mK=0$ by Theorem 25.9 of \cite{HR1}.  Then the closed subgroup $G_m$ of $G$ is contained in the kernel, $\ker(f)$, of $f$. So $f$ induces an injective continuous homomorphism ${\tilde f}$ from $G/\ker(f) = \big(G/G_m\big)/\big(\ker(f)/G_m\big)$ onto $K$. As $G/G_m $ is a compact group, we obtain that $G/\ker(f)$ is also compact. Hence ${\tilde f}$ is a topological group isomorphism of the compact group $G/\ker(f)$ onto $K$. Since the projection $\pi: G\to G/\ker(f)$ is an open mapping, we obtain that $f= {\tilde f} \circ \pi$ is  also an open mapping, as required.
\end{proof}

We now recall two algebraic notions.  An abelian group $G$ is called {\em reduced} if $G$ does not have non-trivial divisible subgroups. Clearly, \emph{the group of all integers, $\ZZ$, is reduced}. An abelian group $G$ is called {\em algebraically compact} if $G$ is a direct summand of an abelian group which admits a compact  group topology (see the Corollary in \cite{Balcerzyk}). 

To prove Theorem \ref{t:H-M-main} we need the following lemma which is an immediate corollary of Theorem 6.4 of \cite{LeggWalker}.

\begin{lemma}\label{l:H-M-integers} The group $\ZZ$ is not algebraically compact.
\end{lemma}

\medskip
Now we prove our main result.

\begin{theorem} \label{t:H-M-main}
Let $K$ be a pro-Lie group which has an open almost connected subgroup $H$.  For every cardinal $\kappa$, any surjective continuous homorphism $f: \ZZ^\kappa \to K$ is an  open mapping.
\end{theorem}

\begin{proof}
Without loss of generality we shall assume that the group $H$ is infinite and hence the cardinal $\kappa$ is also  infinite. We split the proof into two steps.

{\em Step 1.} Assume that $K$ is an almost connected pro-Lie group.
By Lemmas \ref{l:H-M-disconnected} and \ref{l:H-M-tot-dis}, we can assume also that $K$ is compact. It is enough to prove that the image $S:=f(U)$ of an open subgroup $U=\{ 0_i\} \times \ZZ^{\kappa\setminus\{ i\}}$ of $\ZZ^{\kappa}$ is open in $K$, for every $i\in\kappa$.

Set $e:=f(1_i)\in K$ and let $\langle e\rangle$ be the cyclic subgroup of $K$ generated by $e$. Note that, by hypothesis, $K=\langle e\rangle + S$. We have to show that $S$ is open.

We claim that there is an $m\in\NN$ such that $me \in S$. Suppose this is not the case, then we obtain that $\langle e\rangle \cap S=\{0\}$, and hence the subgroup $\langle e\rangle \cong \ZZ$ is a direct (algebraic) summand of the compact group $K$. So $\ZZ$ is an algebraically compact group which is false since it contradicts Lemma \ref{l:H-M-integers}.

So let $m\in\NN$ be such that $me \in S$. Then $mK \subset S$. Let $\pi: K\to K/mK$ be the quotient map. Since $K/mK$ is torsion, Lemma \ref{l:H-M-bounded} implies that the map ${\bar f}:= \pi\circ f$ is open. So ${\bar f}(U)$ is open in $K/mK$. Hence the subgroup $f(U)=S=\pi^{-1} \big( {\bar f}(U)\big)$ is open in $K$. Thus $f$ is an open mapping.

{\em Step 2.} Assume that $K$ contains an open almost connected subgroup $H$. Since the subgroup $X:=f^{-1}(H)$ of $\ZZ^\kappa$ is open, we can find a finite subset $F=\{ i_1,\dots, i_n\}$ of $\kappa$ such that $X$ contains the open subgroup $Y:= \ZZ^{\kappa\setminus F}$. Since $X/Y$ is a subgroup of $\ZZ^n =\ZZ^\kappa/Y$,  there is a $k\in\NN$ such that $X/Y =\ZZ^k$ by \cite[A 26]{HR1}.

As the projection $\pi_Y$ of $X$ onto $Y$ is continuous and $\pi_Y(y)=y$, for every $y\in Y$, we obtain that $X=X/Y \times Y$, see \cite[6.22]{HR1}. So $X$ is topologically isomorphic to $\ZZ^k\times Y$. Hence the restriction map $p:= f|_{X}$ from $X$ onto $H$ is open by Step 1.  As $H$ is open, we obtain that  $f$ is also an open mapping, as required.
\end{proof}

The Principal Structure Theorem for Locally Compact Abelian Groups  (Theorem 25 of \cite{Morris}) says that every locally compact abelian group $K$ has an open subgroup $H$ which is topologically isomorphic to $\RR^n\times C$, where $C$ is a compact abelian group and $n$ is a non-negative integer. So $H$ is an almost connected pro-Lie grup. So as an immediate consequence of Theorem \ref{t:H-M-main} we obtain Corollary \ref{main-corollary}, which provides a positive answer to Question 5 of \cite{HM}.

\begin{corollary}\label{main-corollary} 
Let $K$ be a locally compact abelian group.  For every cardinal $\kappa$, any surjective continuous homorphism $f: \ZZ^\kappa \to K$ is an  open mapping. In particular, this is the case if $K$ is compact.
\end{corollary}

Indeed, since a pro-Lie group $K$ with the property that $K/K_0$ is locally compact has an open  subgroup which is an almost connected pro-Lie group by Corollary 8.12  of \cite{HM-Structure}, we obtain the stronger result:

\begin{corollary} 
Let $K$ be  an abelian pro-Lie group $K$ with the property that $K/K_0$ is locally compact.  For every cardinal $\kappa$, any surjective continuous homorphism $f: \ZZ^\kappa \to K$ is an  open mapping.
\end{corollary}

We conclude with an open question.

\begin{question}
Is every surjective continuous homomorphism from $\ZZ^\kappa$ onto a pro-Lie group $K$ open? 
\end{question}

\end{document}